\documentclass[12pt]{article}
\usepackage{graphicx,psfrag,amssymb,amsmath,amsthm,amsfonts}
\usepackage{tikz, comment}
\usepackage{pgfplots}
\usepackage{tikz}
\usetikzlibrary{arrows,backgrounds}
\pgfdeclarelayer{bkgr}
\pgfdeclarelayer{bkgrd}
\pgfsetlayers{bkgrd,bkgr,main}
\usetikzlibrary{arrows}
\usetikzlibrary{decorations.pathmorphing}
\usepackage{bbm}
\usepackage{hyperref}
\usepackage{color}
\usepackage{framed}
\usepackage{mathrsfs}
\usepackage{mathtools}
\usepackage{eucal}
\def\beq{\begin{equation}}
\def\eeq{\end{equation}}

\def\dist{{\rm dist}}

\def\bh{\textbf{h}}

\def\bJ{\mathbf{J}}

\def\bN{\mathbf{N}}

\def\bR{\mathbf{R}}

\def\bf{\textbf{}}

\def\bh{\mathbf{h}}

\def\br{\mathbf{r}}
\def\bs{\mathbf{s}}

\def\bx{\mathbf{x}}

\def\b1{{\boldsymbol{1}}}

\def\cA{\mathcal{A}}
\def\cB{\mathcal{B}}
\def\cC{\mathcal{C}}

\def\cG{\mathcal{G}}
\def\cH{\mathcal{H}}

\def\cJ{\mathcal{J}}\def\eps{\varepsilon}

\def\cM{\mathcal{M}}

\def\cO{\mathcal{O}}

\def\cT{\mathcal{T}}

\def\cW{\mathcal{W}}

\def\cZ{\mathcal{Z}}
\def\N{\ensuremath{\bf N}}

\def\fF{\mathfrak{F}}

\def\eps{\varepsilon}

\def\dist{\text{\rm dist}}

\def\I{\mathbbm{1}}

\newtheorem{theorem}{Theorem}

\newtheorem{lemma}[theorem]{Lemma}
\newtheorem{proposition}[theorem]{Proposition}

\newtheorem{remark}{Remark}
\newtheorem{defn}{Definition}

\usepackage{mathtools}
\newcommand{\set}[1]{\{#1\}}
\newcommand{\A}[1]{A^{(q)}}
\newcommand{\abs}[1]{\left\lvert#1\right\rvert}
\newcommand\floor[1]{\lfloor#1\rfloor}

\usepackage[T2A,T1]{fontenc}
\makeatletter
\newcommand*{\math@version@bold}{bold}
\DeclareMathOperator\DD{
	\textrm{%
		\usefont{T2A}{cmr}{\ifx\math@version\math@version@bold bx\else m\fi}{n}%
		\CYRD
	}%
}
\makeatother

\begin{document}
\title{Compound Poisson law for  hitting times to periodic orbits in two-dimensional hyperbolic systems.}
\author{Meagan Carney \thanks{Department of Mathematics, University of Houston, Houston, USA. e-mail: $<$meagan@math.uh.edu$>$
Meagan Carney thanks the NSF for partial support on NSF-DMS Grant 1600780.}
\and Matthew Nicol \thanks{Department of Mathematics,
University of Houston,
Houston Texas,
USA. e-mail: $<$nicol@math.uh.edu$>$.
 Matthew Nicol thanks the NSF for partial support on NSF-DMS Grant 1600780 and the hospitality and
 support of the Max Planck Institute for the Physics of Complex Systems, Dresden, where this work was completed.
}\and  Hong-Kun Zhang
 }

\date{\today}

\maketitle

\begin{abstract}
We show that a compound Poisson distribution holds for scaled exceedances of observables $\phi$ uniquely maximized at a periodic point $\zeta$  in a variety of two-dimensional hyperbolic dynamical  systems with singularities $(M,T,\mu)$, including the billiard maps of Sinai dispersing billiards in both the finite and infinite horizon case.  The observable we consider is of form $\phi (z)=-\ln d(z,\zeta)$ where $d$ is a metric defined in terms of the stable and unstable foliation.
The compound Poisson process we obtain  is a
P\'olya-Aeppli distibution of index $\theta$. We calculate $\theta$ in terms of the derivative of the map $T$.  Furthermore if we define $M_n=\max\{\phi,\ldots,\phi\circ T^n\}$
and $u_n (\tau)$ by $\lim_{n\to \infty} n\mu (\phi >u_n (\tau) )=\tau$ the maximal process satisfies an extreme value law of form $\mu (M_n \le u_n)=e^{-\theta \tau}$.
These results generalize to a broader class of functions maximized at $\zeta$, though the formulas regarding the parameters in the distribution need to be modified.

\end{abstract}

\centerline{AMS classification numbers: 37D50, 37A25}

\section{Motivation and relevant works}

The study of the statistical properties of 2-dimensional hyperbolic
systems with singularities was  motivated in large part by a desire to understand mathematical
billiards with chaotic behavior. This model was introduced by Sinai in~\cite{Sinai} and
has been studied  by many authors~\cite{BSC90, BSC91, Y98,Y99,CM}.

The statistical properties of chaotic dynamical systems are often described by probabilistic  limit theorems. Let
$(M,T,\mu)$ be a dynamical system, i.e., a transformation
$T\colon M\to M$ preserving a probability measure $\mu$ on $M$.
For any real-valued function $\phi$ on $M$ (often called an
observable), let $X_n=\phi\circ T^n$, for any $n\geq 0$. We consider the stochastic process generated by the time-series $\{X_n\}$. We assume $\phi$ has a unique maximal point $\zeta$. For any fixed high threshold $u>0$,  rare events can be defined  as the event of  the exceedance $(X_n>u)$, for some $n\geq 0$.   Laws of rare events for chaotic dynamical systems have been investigated, with results on both   hitting time statistics (HTS) and return time statistics (RTS)~\cite{Collet,DGS,Hirata,haydn-wasilewska,Gupta,CC,FFT1,FFT2}.  Techniques from extreme value theory (EVT) have been used to understand the
statistics of rare events~\cite{Collet, FFT1,FFT2,Gupta,GHN,HNT,Keller,Ferguson_Pollicott}. Extreme value theory concerns the distributional and almost sure limits  of the derived time series of maxima $M_n:=\max\{X_1, \cdots, X_n\}$. Since we are assuming $\phi$ has a unique maximum at $\zeta$ there is a close relation between extreme values of $\phi \circ T^n$ and visits to shrinking neighborhoods of $\zeta$.

Let $u_n$ be a sequence of constants defined by the requirement that $\lim_{n\to\infty} n\mu(\psi>u_n)=\tau$. The research cited above has shown that for a variety of  chaotic dynamical systems, with this scaling,   $$\lim_{n\to\infty} \mu(M_n\leq  u_n)=e^{-\theta \tau},$$ where $\theta\in [0,1]$.  The parameter $\theta$ is called the extremal index and roughly measures the clustering of exceedances of the maxima. In fact $\frac{1}{\theta}$ is the average cluster size of exceedances given that one exceedance occurs.

 For certain one-dimensional uniformly expanding maps and Anosov toral automorphisms a strict dichotomy has been observed. The dichotomy is that
 $\theta=1$ if $\zeta$ is not periodic and $\theta<1$ otherwise~\cite{Keller,Ferguson_Pollicott,FFT2}. In this paper, we investigate the extremal index when $\zeta$ is a periodic point in the setting of
 two-dimensional hyperbolic systems with singularities. Poisson return time statistics for  generic points in a variety of billiard systems (both polynomially and exponentially mixing) were established in~\cite{FHN}. Related results on Poisson return time statistics
 were obtained for Young Towers with polynomial tails in~\cite{Pene_Saussol}. However the results of  both works~\cite{FHN, Pene_Saussol} were limited to a full measure set of generic points, which explicitly  excluded  periodic orbits. This paper extends the results of~\cite{FHN} to the case of periodic points in a large class of billiard systems, using somewhat different techniques. For periodic points we show a compound Poisson distribution for exceedances holds and that the extremal index is determined by the Jacobian of the map $T$ along the
 periodic orbit (a precise formula is given in the statement of our main theorem). In~\cite{Dichotomy} a similar
 result was shown for smooth toral automorphisms, and in fact in that setting a strict dichotomy was shown ($\theta=1$ if $\zeta$ is not periodic and $\theta<1$ otherwise). We don't expect a dichotomy to hold in the class of  hyperbolic systems with singularities,
for example chaotic billiards.  We expect the presence of singularities to enable a variety of extremal
distributions, for example at observations maximized on singular sets. The  analysis of  the statistical properties of hyperbolic systems with singularities is more difficult than for smooth uniformly hyperbolic systems and the arguments of~\cite{Dichotomy}
do not generalize in a straightforward way.
The main difficulty is caused by the
singularities and the resulting fragmentation of phase space  during the
dynamics, which slows down the global expansion of unstable manifolds. We are able to use a growth lemma~\cite{CM} to overcome these difficulties
under certain other  assumptions on the dynamics. Our results apply in particular to the billiard mapping of  Sinai dispersing billiards with finite and infinite horizon.

\subsection{Main assumptions}

Let $M$ be a 2-dimensional compact Riemannian manifold, possibly with
boundary. Let $\Omega \subset M$ be an open subset and let $T\colon
\Omega \to M$ be a $C^{1+\gamma}$ diffeomorphism of $\Omega$ onto
$T(\Omega)$ (here $\gamma\in (0,1]$). We assume that $S_1 =
M\setminus\Omega$  is a finite or countable union of smooth compact
curves . Similarly, $S_{-1} = M\setminus T(\Omega)$  is a finite or
countable union of smooth compact curves. If $M$ has boundary
$\partial M$, it must be a subset of both $S_1$ and $S_{-1}$. We
call $S_1$ and $S_{-1}$ \emph{singularity sets} for the maps $T$ and
$T^{-1}$, respectively. We denote by $\Omega_i$, $i\geq 1$, the connected components of
$\Omega$; then $T(\Omega_i)$ are the connected components of
$T(\Omega)$. We assume that $T|_{\Omega}$ is time-reversible, and the restriction of the map $T$ to any
component $\Omega_i$ can be extended by continuity to its boundary
$\partial \Omega_i$, though the extensions to $\partial \Omega_i \cap
\partial \Omega_j$ for $i\neq j$ need not agree. Similarly, for each $i$ the
restriction of $T^{-1}$ to any connected component $T(\Omega_i)$
can be extended by continuity to its boundary $\partial T(\Omega_i)$.

Next we assume that the map $T$ is  hyperbolic, as defined by Katok and Strelcyn~\cite{KS}. This means that $T$ preserves
a probability measure $\mu$ such that $\mu$-a.e. point $x\in M$ has
two non-zero Lyapunov exponents: one positive and one negative. Also,
the first and second derivatives of the maps $T$ and $T^{-1}$ do
not grow too rapidly near their singularity sets $S_1$ and $S_{-1}$,
respectively, and the $\epsilon$-neighborhood of the singularity set has measure $\cO(\epsilon)$; this is to ensure the existence and absolute continuity of stable
and unstable manifolds at $\mu$-a.e. point.
Let $\cW^u=\cap_{n\geq 0} T^n (\cM\setminus S_1)$.
$\cW^u$ is (mod 0) the union of all unstable manifolds, and we assume
that the partition  $\cW^u$ of $M$ into unstable manifolds is measurable,
so that $\mu$ induces conditional distributions on $\mu$-almost all
unstable manifolds (see the definition and basic properties of
conditional measures in \cite[Appendix~A]{CM}). Most importantly, we
assume that the conditional distributions of $\mu$ on unstable
manifolds $W\subset \cW^u$ are absolutely continuous with respect to
the Lebesgue measure on $W$. This means that $\mu$ is the so called
Sinai-Ruelle-Bowen (SRB) measure.

We also assume that our SRB
measure $\mu$ is ergodic and mixing.
 In physics terms, $\mu$ is an
equilibrium state for the potential $-\ln DT_{|\cW^u}$.


For $n\geq 1$, let
$$
    S_{n}=\cup_{i=0}^{n-1} T^{-i} S_1 \,\,\,\,\text{ and }\,\,\,\,\, S_{-n}=\cup_{i=0}^{n-1} T^i S_{-1},
$$
for each $n \geq  1$. Then the map $T^n\colon M\setminus  S_{n}\to
M\setminus  S_{-n}$ is a  $C^{1+\gamma_0}$ diffeomorphism.\\

We next make more specific assumptions on the system $(M,T,\mu)$ to give
sufficient conditions for  exponential decay rates of correlations, as well as for the coupling lemma. These assumptions have been made in other works in the literature~\cite{C99,CD,CM,CZ09}.\\


\begin{itemize}
\item[(\textbf{h1})] \textbf{Hyperbolic cones for $T$\footnote{We have already assumed that Lyapunov exponents are not zero a.e., but our methods also use stable and unstable cones for the map $T$.} of $T$}. There exist two families of cones
$C^u_x$ (unstable) and $C^s_x$ (stable) in the tangent spaces
${\cal T}_x M$, for all $x\in M\setminus S_1$, and there exists a
constant $\Lambda>1$, with the following properties:
\begin{itemize}
\item[(1)] $D_x T (C^u_x)\subset C^u_{ T x}$ and $D_x T
    (C^s_x)\supset C^s_{ T x}$, wherever $D_x T $ exists.
    \item[(2)] $\|D_x T(v)\|\geq \Lambda \|v\|, \forall
    v\in C_x^u, \quad\text{and}\quad   \|D_xT^{-1}(v)\|\geq
    \Lambda \|v\|, \forall v\in C_x^s$. \item[(3)] These
    families of cones are   continuous on $M$
 and the angle between $C^u_x$ and $C^s_x$ is uniformly
 bounded away from
zero.
 \end{itemize}

We say that a smooth curve $W\subset M$ is an unstable (stable) \emph{curve} if at every point $x \in W$ the tangent line
$\cT_x W$ belongs in the unstable (stable) cone $C^u_x$ ($C^s_x$).
Furthermore, a curve $W\subset  M$ is an unstable (stable) \emph{manifold} if $T^{-n}(W)$ is an unstable (stable) curve for all $n \geq 0$ (resp. $\leq 0$).

\item[(\textbf{h2})] \textbf{Singularities.} The boundary $\partial M$ is transversal to both stable and
    unstable cones. Every other smooth curve $W\subset S_1\setminus \partial M$ (resp.
    $W\subset S_{-1}\setminus \partial M$ ) is a
 stable (resp.  unstable) curve. Every curve in $ S_1$
 terminates either inside another curve of $ S_1$ or on
 the boundary $\partial M$. A similar assumption is made for $S_{-1}$. Moreover, there exists  $C>0$ such
    that for any $x\in M\setminus  S_1$ \beq\label{upper}
   \|D_x T \|\leq C\, \dist(x,  S_1)^{-1},\eeq
   and for any $\epsilon >0$,
   \beq\label{epscs11}\mu\bigl(x\in M\colon \dist(x, S_1)<\eps\bigr)<C\epsilon.\eeq

Note that (\ref{epscs11}) implies that for $\mu$-a.e. $x\in M$, there exists a stable and unstable manifold $W^{u/s}(x)$, such that $T^n W^s(x)$ and $T^{-n}W^u(x)$ does not intersect $ S_1$, for any $n\geq 0$.

\begin{defn}
For every $x, y \in M$, define $\bs_+(x,y)$, the forward
\emph{separation time} of $x, y$, to be the smallest integer
$n\geq 0$ such that $x$ and $y$ belong to distinct elements of
$M\setminus S_n$.
 Fix $\beta\in(0,1)$, then $d(x,y)=\beta^{\bs_+(x,y)}$
 defines a metric on $ M$.
Similarly we define the backward separation time $\bs_-(x,y)$.
\end{defn}

\item[(\textbf{h3})] \textbf{Regularity of stable/unstable
curves}. We assume that the following families of stable/unstable curves, denoted by $\cW^{s,u}_T$ are invariant under $T^{-1}$ (resp., $T$) and include all stable/unstable manifolds:

\begin{enumerate}
\item[(1)]  \textbf{Bounded curvature.} There exist $B>0$ and $c_M>0$, such that the  curvature
    of any $W\in \cW^{s,u}_T$ is uniformly bounded from above by $B$, and the length of the curve $|W|<c_M$.
 \item[(2)] \textbf{Distortion bounds.} There exist $\gamma_0\in (0,1)$ and   $C_{\br}>1$ such
    that for any unstable curve $W\in \cW^{u}_T$ and
    any $x, y\in W$,
 \beq
      \left|\ln\cJ_W (T^{-1}x)-\ln \cJ_W (T^{-1}y)\right|       \leq C_{\br}\, \dist(x, y)^{\gamma_0} \label{distor10}
 \eeq
where  $\cJ_W (T^{-1}x)=dm_{T^{-1}W}(T^{-1}x)/dm_W(x)$  denotes  the Jacobian
of $T^{-1}$ at $x\in W$ with respect to the Lebesgue measure $m_W$ on the unstable curve $W$.

\item[(3)] { \textbf{Absolute continuity.}}
 Let $W_1,W_2\in \cW^{u}_T$ be two unstable curves close to each other. Denote
 \begin{equation*}
 W_i'=\{x\in W_i\colon
W^s(x)\cap W_{3-i}\neq\emptyset\}, \hspace{.5cm} i=1,2.
 \end{equation*} The map
$\bh\colon W_1'\rightarrow W_2'$ defined by sliding along stable
manifolds
 is called the \textit{holonomy} map. We assume $\bh_*\mu_{W_1'} \prec \mu_{W_2'}$, and furthermore, there exist uniform constants $C_{\br}>0$ and $\vartheta_0\in (0,1)$, such that  the Jacobian of $\bh $ satisfies
 \beq\label{Jh}
 |\ln\cJ\bh(y)-\ln \cJ\bh(x)| \leq C_{\br}
 \vartheta_0^{\bs_+(x,y)},
\hspace{1cm}\forall x, y\in W_1' \eeq
Similarly, for any $n\geq 1$ we can define the holonomy map
$$
   \bh_n=T^n\circ \bh \circ T^{-n}:T^n W_1\to T^n W_2,
$$
and then (\ref{Jh}) and the uniform hyperbolicity (\textbf{h1}) imply
\beq \label{cJhn}
    \ln \cJ\bh_n(T^n \bx)\leq C_{\br} \vartheta_0^{n}
\eeq
\end{enumerate}

\item[(\textbf{h4})] {\textbf{ One-step expansion.}}
We have \beq
  \liminf_{\delta\to 0}\
   \sup_{W\colon |W|<\delta}\sum_{n}
  \frac{|T^{-1}V_{n}|}{|V_{n
   }|}<1,
      \label{step1} \eeq where the supremum is taken over regular unstable curves $W\subset M$, $|W|$ denotes the length of $W$, and $V_{n}$, $n\geq 1$, denote the
      smooth components of $T(W)$.

   \end{itemize}

\begin{remark}
These assumptions $\textbf{h1}-$$\textbf{h4}$ are satisfied by the billiard map associated to Sinai dispersing billiards with finite and infinite horizon. Further
systems which satisfy these assumptions are given in the last section on applications.
\end{remark}

\begin{remark}
These assumptions $\textbf{h1}-$$\textbf{h4}$ along with the Growth Lemma imply the existence of a Young Tower with exponential tails~\cite[Lemma 17]{CZ09}.

\end{remark}

  \subsection{Statement of the main results}

 We fix a  hyperbolic periodic point $\zeta\in M$ with prime period $q>1$, and let $\phi:M\to \mathbb{R}\cup\{+\infty\}$ be given by $$\phi(z)=-\ln (d(z,\zeta)),$$
where $d$ is a metric on $M$ that will be specified later. We assume that iterates of $\zeta$ do not lie  on the singular sets $S_1\cup S_{-1}$.
Our metric $d$ will be adapted to the chart given by the stable and unstable manifolds of $\zeta$, denoted as $W^s(\zeta)$ and $W^u(\zeta)$. If the stable manifold  of $x$, denoted as $W^s(x)$ intersects $W^u(\zeta)$, say at a point $z$, then we define $x^s:=\dist_{W^u(\zeta)}(z, \zeta)$, i.e. the distance of $z$ and $\zeta$ measured along the unstable manifold $W^u(\zeta)$.  Note that if $W^u(x)$ is very short and does not reach $W^s(\zeta)$, then we may extend $W^u(x)$ as an unstable curve, thus $x^u$ can be similarly  defined. If the unstable manifold (or extension to an unstable curve) of $x$, denoted as $W^u(x)$ intersects $W^s(\zeta)$ at $z$, we define $x^u:=\dist_{W^s(\zeta)}(z, \zeta)$.  The foliation of stable and unstable curves of a sufficiently small neighborhood of $\zeta$ will be H\"older continuous.   Moreover,  if both $x$, $y$
lie in the same local chart determined by  stable and unstable manifolds (or stable and unstable curves) so that $x=(x^u, x^s)$, $y=(y^u, y^s)$,  we define \beq\label{defn:d}d(x,y)=\max\{ |x^u-y^u|,|y^s-y^s| \}.\eeq
As in~\cite{Dichotomy} we expect the form
of the distribution for cluster size in the compound Poisson distribution we obtain to depend upon the metric used. With the dynamically adapted metric we use we obtain a geometric distribution
with parameter equal to  the extremal index $\theta$, $\theta=1-\frac{1}{|DT^q_u (\zeta)|}$, where $DT^q_u (\zeta)$ is the derivative of $T^q$ in the unstable direction at $\zeta$.  With the usual Euclidean metric $\rho$ we would expect a different distribution along the lines of the (complicated) calculations in~\cite[Section 6]{Dichotomy}.

The functional form of $\phi$ will
determine the scaling constants in the extreme value distribution but results for one functional type are easily transformed into any other, see for example~\cite{GHN,FFT1}.
Let $X_n:=\phi\circ T^n$. Since $\mu$ is invariant  the process $\{X_n, n\geq 0\}$ is stationary. Note that $\{X_n>u\}=T^{-n}(B_{e^{-u}}(\zeta))$, where $B_{e^{-u}}(\zeta)$ denotes the $d$-ball centered at $\zeta$ of  radius $e^{-u}$.
In order to obtain a nondegenerate limit for the distribution of $M_n=\max\{X_0,\cdots, X_n\}$, we choose a normalizing sequence $\{u_n\}=\{u_n (\tau)\}$  such that
\[
u_n=u_n(\tau):=\inf \{u>0\,:\, \mu(X_0\leq u)\geq 1-\frac{\tau}{n}\}
\]
for any $\tau>0$. We also define
\[
U_n (\tau)=(X_0>u_n (\tau) )
\]
Then we can check that
\beq\label{chooseun}\lim_{n\to\infty} n \mu(X_0>u_n)=\lim_{n\to\infty} n \mu(U_n)=\tau
\eeq

The extremal index, written as $\theta$, takes values in $[0, 1]$, and can be interpreted
as an indicator of extremal dependence, with $\theta = 1$ indicating
asymptotic independence of extreme events. On the other hand $\theta=0$ represents
the case where we can expect strong clusterings of extreme events.
In this case it is natural to expect excursions away from  moderate values
to extreme regions to persist for random times which have heavy-tailed
distributions.

 Note that the choice of $u_n$ is made so that the mean number of exceedances by $X_0,\cdots, X_n$ is approximately constant. For the case when $\{X_n\}$ are iid, it was shown in~\cite{LLR} that (\ref{chooseun}) is equivalent to the fact that $\mu(M_n\leq u_n)\to e^{-\tau}$ as $n\to\infty$.
We are interested in knowing if there is a non-degenerate distribution $H$ such that the scaled process $\mu (M_n \le u_n)$  converges in distribution to $H$ as $n$ goes to infinity. Then we say that $H$ is an Extreme Value Distribution for $M_n$. Note that by the Birkhoff Ergodic Theorem, we know that $M_n\to \max \{ \phi(\zeta),\infty\}$ almost surely and hence we must scale $u_n\to \max\{ \phi(\zeta),\infty\} $ as $n\to\infty$. We refer the event $\{X_j>u\}$ as an exceedance at time $j$ of level $u$.

The assumption that $\zeta$ is a hyperbolic $q$-periodic point, the assumptions (\textbf{h1})-(\textbf{h4}) and the definition of the $d$-metric  implies that there exists $\theta\in (0,1)$, such that
\beq\label{theta}
\lim_{n\to\infty} \mu(X_q> u_n| X_0>u_n)=1-\theta
\eeq
$\theta=1-\frac{1}{|DT^q_u (\zeta)|}$, where $DT^q_u (\zeta)$ is the derivative of $T^q$ in the unstable direction at $\zeta$. The calculation is the same as in~\cite[Section 6]{Dichotomy}.

Our first theorem is

\begin{theorem}
Let $\zeta$ be a periodic orbit of period $q$, $q\not \in S_{1}\cup S_{-1}$,
and define
 $$\phi(z)=-\ln (d(z,\zeta)),$$
where
 $d$ is the  metric adapted to the chart given by the stable and unstable foliation.
 Define  $M_n:=\max\{\phi , \ldots, \phi\circ T^{n-1}\}$.
Then
\[
\lim_{n\to\infty}\mu(M_n\leq u_n)= e^{-\theta\tau},
\]
where $\theta=1-\frac{1}{|DT^q_u (\zeta)|}$.

\end{theorem}

Now we consider the related rare event point processes (REPP). Let $N_n$ be the number of exceedances of a level $u_n$ by the process $X_0, \cdots, X_n$. If $\{X_n\}$ is an  i.i.d. sequence, it was proved in~\cite{LLR} that if $u_n$ satisfies $n\mu(X_0>u_n)\to \tau$ as $n\to\infty$, then $N_n$ follows a Poisson distribution with intensity $\tau$:
$$\mu(N_n\leq k)\to e^{-\tau}\sum_{i=0}^k \frac{\tau^i}{i!},$$ for any $k\geq 0$. On the other hand, for a stationary process $\{X_n\}$ with clustering of recurrence we expect that $N_n$ converges to a compound Poisson distribution. This behavior has been observed in a variety of one-dimensional expanding maps~\cite{FFT2,Ferguson_Pollicott,Keller} as well as in two-dimensional hyperbolic linear  toral automorphisms~\cite{Hirata,DGS,Dichotomy}. In fact in these papers a strict dichotomy has been observed in the extreme value statistics of functions maximized at
periodic as opposed to non-periodic points.  In this paper we show that a compound Poisson process governs the return times to periodic points for  a large class of two-dimensional invertible hyperbolic systems. As we mentioned the precise form of the compound Poisson process depends upon the metric used. For simplicity we consider a dynamically adapted metric $d$, which yields a geometric distribution for the cluster size of exceedances with parameter $\theta$. Other distributions for  the cluster size would be expected with different metrics, as in~\cite[Section 6]{Dichotomy}.

To make precise our statements we need the following definitions. Denote by $\mathcal{R}$ the ring generated by intervals of the type $[a,b)$, for $a,b\in\bR_0^+$. Thus for every $J\in\mathcal{R}$
there are $k$ intervals $I_1,\ldots,I_k\in\mathcal{S}$, say $I_j=[a_j,b_j)$ with $a_j,b_j\in\bR_0^+$, such that
$J=\cup_{i=1}^k I_j$. For
$I=[a,b)$ and $\alpha\in \bR$, we set $\alpha I:=[\alpha
a,\alpha b)$ and $I+\alpha:=[a+\alpha,b+\alpha)$. Similarly, for
$J\in\mathcal{R}$, we define $\alpha J:=\alpha I_1\cup\cdots\cup \alpha I_k$ and
$J+\alpha:=(I_1+\alpha)\cup\cdots\cup (I_k+\alpha)$.

\begin{defn}\label{def:rare event process}
Given $J\in\mathcal{R}$ and a sequence $(u_n)_{n\in\bN}$, the \emph{rare event point process} (REPP) is defined by
counting the number of exceedances (or events $(X_i>u_n)$) during the (re-scaled) time period $u_n^{-1} J\in\mathcal{R}$. More precisely, for every $J\in\mathcal{R}$, set
\begin{equation}
\label{eq:def-REPP} N_n(J):=\sum_{i\in u_n^{-1} J\cap\N_0}\I_{X_i>u_n}.
\end{equation}
\end{defn}


For the sake of completeness, we also define what we mean by a Poisson and a compound Poisson process.

\begin{defn}
\label{def:compound-poisson-process}
Let $Y_1, Y_2,\ldots$ be  an iid sequence of random variables with common exponential distribution of mean $1/\vartheta$. Let  $Z_1, Z_2, \ldots$ be another iid sequence of random variables, independent of $(Y_i)$, and with distribution function $G$. Given these sequences, for $J\in\bR$ define
$$
N(J)=\int \I_J\;d\left(\sum_{i=1}^\infty Z_i \delta_{Y_1+\ldots+Y_i}\right),
$$
where $\delta_t$ denotes the Dirac measure at $t>0$.  We say that $N$ is a compound Poisson process of intensity $\vartheta$ and multiplicity d.f.\ $G$.
\end{defn}
\begin{remark}
In our context we think of the random variable  $Z_i$ as determining the number of exceedances given that one has occurred in a certain time interval.
\end{remark}

\begin{remark}
\label{rem:poisson-process}
In our setting $Z_i$ will always be integer valued, as it gives the size of a cluster. This means that $\Lambda$ is completely defined by the values $G(k)=P(Z_1=k)$, for every $k\in\N_0$. Note that, if $G(1)=1$ and $\vartheta=1$, then $N$ is the standard Poisson process and, for every $t>0$, the random variable $N([0,t))$ has a Poisson distribution of mean $t$.
\end{remark}

\begin{remark}
\label{rem:compound-poisson}
In this paper $G$ turns out to be a  geometric distribution of parameter $\theta \in (0,1]$,  i.e. $G(k)=\theta(1-\theta)^{k-1}$, for every $k\in\N_0$. This means that, as in~\cite{HV09,Dichotomy}, the random variable $N([0,t))$ follows a P\'olya-Aeppli distribution
$$
P(N([0,t)))=k=e^{-\theta t}\sum_{j=1}^k \theta^j(1-\theta)^{k-j}\frac{(\theta t)^j}{j!}\binom{k-1}{j-1},
$$
for all $k\in\N$, and that $P(N([0,t))=0)=e^{-\theta t}$. So there is just one parameter $\theta$ governing the compound Poisson distribution.
\end{remark}

\begin{theorem}
\label{thm:dichotomy4}
Let $\zeta$ be a periodic orbit of period $q$, $q\not \in S_1\cup S_{-1}$,
and
 $$\phi(z)=-\ln (d(z,\zeta))$$
where
 $d$ is the  metric adapted to the chart given by the stable and unstable foliation.
 Define $X_n=\phi\circ T^n$ and $M_n:=\max\{X_1,\ldots,X_n\}$. Let $(u_n)_{n\in\N}$ be a sequence satisfying $\lim_{n\to \infty} n\mu (X_0>u_n)=\tau$ for some $\tau\geq 0$ and $(v_n)_{n\in\N}$ be given by $v_n=1/\mu (X_0>u_n)$. Consider the REPP $N_n$ as in Definition~\ref{def:rare event process}. Then the REPP $N_n$ converges in distribution to a compound Poisson process with intensity $\theta=1-\frac{1}{|DT^q_u (\zeta)|}$ and geometric multiplicity d.f. $G^*$ given by $G^*(k)=\theta (1-\theta)^{k -1}$, for all $k\in\N$.

\end{theorem}

\section{Extreme value theory scheme of proof.}

Our proofs are based on ideas from extreme value theory. Let $X_n=\phi\circ T^n$ and define $A_n^q :=\lbrace X_0>u_n,X_q\leq  u_n\rbrace$.
For $s,l \in \mathbb{N}$ and a set $B\subset M$, let
\[
\mathscr{W}_{s,l}(B)=\bigcap_{i=s}^{s+l-1} T^{-i}(B).
\]
Next we introduce two conditions following \cite{Dichotomy}.\\

\noindent {\textbf{ Condition $\DD_q(u_n)$}}: We say that $\DD_q(u_n)$ holds for the sequence $X_0,X_1,\ldots$ if, for every  $\ell,t,n\in \mathbb{N}$
\[
\left|\mu\left(A_n^q\cap
  \mathscr{W}_{t,\ell}\left(A_n^q\right) \right)-\mu\left(A_n^q\right)
  \mu\left( \mathscr{W}_{0,\ell}\left(A_n^q\right)\right)\right|\leq \gamma(q,n,t),
\]
where $\gamma(q,n,t)$ is decreasing in $t$ and  there exists a sequence $(t_n)_{n\in \mathbb{N}}$ such that $t_n=o(n)$ and
$n\gamma(q,n,t_n)\to0$ when $n\rightarrow\infty$.\\

 We consider the sequence $(t_n)_{n\in\N}$ given by condition $\DD_q(u_n)$ and let $(k_n)_{n\in\N}$ be another sequence of integers such that as $n\to\infty$,
\[
k_n\to\infty\quad \mbox{and}\quad  k_n t_n = o(n).
\]

\noindent {\textbf{ Condition $\DD'_q(u_n)$}}:  We say that $\DD'_q(u_n)$
holds for the sequence $X_0,X_1,\ldots$ if there exists a sequence $(k_n)_{n\in\N}$ as above  and such that
\[
\lim_{n\rightarrow\infty}\,n\sum_{j=1}^{\lfloor n/k_n\rfloor}\mu\left( A_n^q\cap T^{-j}\left(A_n^q\right)
\right)=0.
\]

We note that, when $q=0$, which corresponds to a non-periodic point, condition $\DD'_q(u_n)$ corresponds to condition $D'(u_n)$ from~\cite{LLR}.

Now let
\[
\theta=\lim_{n\to\infty}\theta_n=1-\lim_{n\to\infty}\frac{\mu(A^q_n)}{\mu(U_n)}.
\]

From \cite[Corollary~2.4]{FFT2}, it follows that to proof Theorem 1  it suffices to prove  conditions $\DD_q(u_n)$ and $\DD'_q(u_n)$ since if
both conditions hold then  from~\cite[Corollary~2.4]{FFT2} the limit exists and 
\[
\lim_{n\to\infty}\mu(M_n\leq u_n)= e^{-\theta\tau}.
\]

To establish Theorem~\ref{thm:dichotomy4} it suffices to establish Condition $\DD'_q(u_n)$ and a property related to Condition $\DD_q(u_n)$, which we call Condition $\DD^*_q(u_n)$~\cite{FFT3} (see also the discussion in~\cite[Section 8.1]{Dichotomy}).  We state $\DD^*_q(u_n)$ for completeness, but its proof is a decay of correlations result which requires only very minor modifications to that of the
proof of Condition $\DD_q(u_n)$. We will only prove Condition $\DD_q(u_n)$ in this paper. The proof of
Condition $\DD_q(u_n)$ uses the assumption of a Young Tower to approximate the indicator function of a complicated set
by a function constant on stable manifolds, which allows a decay rate to be bounded by the $L^{\infty}$ norm of
the set rather than a H\"older norm. In fact the scheme of  the proof of Condition $\DD_q(u_n)$ is itself  somewhat standard~\cite{FHN,Dichotomy}  and the novelty of this paper is the proof of Condition $\DD'_q(u_n)$ in this setting, which uses the growth lemma of~\cite{CM} applied to unstable curves.

Let $\zeta$  be a periodic point of prime period $q$.  Recall $U_n= (X_0>u_n)$ and define  the sequence $U^{(k)} (u_n))$ of nested sets  centered at $\zeta$ given by
\begin{equation}
\label{eq:Uk-definition}
U^{(0)}(u_n)=U_n
 \quad\text{and}\quad U^{(k)}(u_n)=T^{-q}(U^{(k-1)}(u_n))\cap U_n, \quad\text{for all $k\in\mathbb{N}$.}
 \end{equation}
For $i,k,l,s\in\mathbb{N}\cup\{0\}$, we define the following sets:
\begin{equation}\label{eq:Q-definition}
Q_{q,i}^k (u_n):=T^{-i}\left(U^{(k)}(u_n)-U^{(k+1)}(u_n)\right).
\end{equation}

Note that $Q_{q,0}^0(u_n)=A_n^{(q)}$.
Furthermore,
$U_n=\bigcup_{k=0}^\infty Q_{q,0}^k (u_n)$.
Using our $d$-metric the set $U_n$ centered at $\zeta$ can be decomposed into a sequence of disjoint strips where $Q_{q,0}^0(u_n)$ are the most outward strips and the inner strips $Q_{q,0}^{k+1}(u_n)$ are sent outward by $T^q$ onto the strips $Q_{q,0}^k (u_n)$, i.e.,
$T^{q}(Q_{q,0}^{k +1}(u_n))=Q_{q,0}^k (u_n).$

\begin{defn}[$D_q(u_n)^*$]\label{cond:Dp*}We say that $D_q(u_n)^*$ holds
for the sequence $X_0,X_1,X_2,\ldots$ if for any integers $t, k_1,\ldots,k_q$, $n$ and
 any $J=\cup_{j=2}^q I_j\in \bR$ with $\inf\{x:x\in J\}\ge t$,
 \[ \left| P \left(Q_{q,0}^{k_1}(u_n)\cap \left(\cap_{j=2}^q N_{n}(I_j)=k_j \right) \right)-P\left(Q_{q,0}^{k_1}(u_n)\right)
  P\left(\cap_{j=2}^q N_n (I_j)=k_j \right)\right|\leq \gamma(n,t),
\]
where for each $n$ we have that $\gamma(n,t)$ is nonincreasing in $t$  and
$n\gamma(n,t_n)\to 0$  as $n\rightarrow\infty$, for some sequence
$t_n=o(n)$.
\end{defn}

\section{Background on billiards: growth lemma and decay of correlations for hyperbolic billiards.}
The use of standard pairs of unstable and stable curves ~\cite{CM} instead of standard pairs of unstable and stable manifolds simplifies our proof of
Condition $\DD_q'(u_n)$. We now sketch some background on the use of standard pairs.

For any $\gamma\geq \gamma_0$, let $\cH^-(\gamma)$	be the set of all  bounded real-valued  functions $\varphi\in L_{\infty}(\cM,\mu)$ such that  for any $x$ and $y$ lying on one stable curve $W^s\in \cW^s_{T}$, such that

$$\|\varphi\|^-_{\gamma}\colon = \sup_{W^s\in \cW^s_{T}}\sup_{ x, y\in W^s}\frac{|\varphi(x)-\varphi(y)|}{\dist(x,y)^{\gamma}}<\infty.$$
and
\beq \label{DHC-} 	|\varphi(x) - \varphi(y)| \leq \|\varphi\|^-_{\gamma} \dist(x,y)^{\gamma}\eeq

Similarly we define $\cH^{+}(\gamma)$ as the set of all real-valued  functions $\psi\in L_{\infty}(\cM,\mu)$  such that  for any $x$ and $y$ lying on one unstable curve $W^u \in \cW^u_{T}$ such that
$$
   \|\psi\|^+_{\gamma}\colon = \sup_{W^u\in \cW^u_{T}}\sup_{ x, y\in W^u}\frac{|\psi(x)-\psi(y)|}{\dist(x,y)^{\gamma}}<\infty,
$$
and
\beq \label{DHC+} 	
   |\psi(x) - \psi(y)| \leq \|\psi\|^+_{\gamma} \dist(x,y)^{\gamma}.
\eeq

When we study autocorrelations of certain observables, we will need to require that the latter belong to the space $\cH(\gamma)\colon =\cH^+(\gamma)\cap\cH^-(\gamma)$, i.e., they are  H\"older functions  on every stable and unstable manifold. For correlations of two distinct functions $\varphi$, $\psi$ as above we always assume that $\varphi\in \cH^-(\gamma_1)$ and $\psi\in \cH^+(\gamma_2)$ with some $\gamma_1,\gamma_2\in[\gamma_0,1]$,  unless otherwise specified.
For every $\varphi\in \cH^{\pm}(\gamma)$ we define
\beq \label{defCgamma}
   \|\varphi\|^{\pm}_{C^{\gamma}}\colon=\|\varphi\|_{\infty}+\|\varphi\|^{\pm}_{\gamma}.
\eeq

By using coupling methods, we obtain the following bounds for the rate of decay of correlations for our class of hyperbolic systems with singularities, see~\cite{CZ09}.

\begin{lemma}\label{lemma:exp}
For systems satisfy (h1)-  (h4), there exist $C>0$, $p\in (0,1)$, such that for any
observables $\varphi\in \cH^-(\gamma_1)$ and $\psi\in\cH^+(\gamma_2)$ on $\cM$, with $\gamma_i\in [\gamma_0,1]$, for $i=1,2$,
\begin{equation}\label{thm:6} |\int \psi \cdot \varphi \circ T^n d\mu  -\int \psi d\mu \int \phi d\mu |\leq C\|\psi\|^+_{_{C^{\gamma_2}}} \|\varphi\|^-_{_{C^{\gamma_1}}} \Lambda^{-pn}\end{equation}
for any $n\geq 1$.
\end{lemma}

Let $C_{\br}>0$ and $\gamma_0\in (0,1)$ be the constants given in (\ref{distor10}).
Fix a large constant $C_{\bJ}>C_{\br}/(1-1/\Lambda)$.
Given an unstable curve $W$ and a probability measure $\nu$ on $W$, the pair
$(W,\nu)$ is said to be a \textit{standard pair} if $\nu$ is absolutely
continuous with respect to $\mu_W$,
such that the density function
$g_W(x):=d\nu/d\mu_W (x)$ is regular in the sense that
 \beq\label{lnholder0}
 | \ln g_W(x)- \ln g_W(y)|\leq C_{\bJ}\cdot d_W(x,y)^{\gamma_0}.
 \eeq

Iterates of standard pairs require the following extension of standard pairs, see \cite{CM, CZ09}.
Let $\cG=\{(W_{\alpha}, \nu_\alpha),\ \alpha\in \cA, \ \lambda\}$ be a family of standard pairs
equipped with a factor measure $\lambda$ on the index set $\cA$.
We call $\cG$ a {\it standard family} if the following conditions hold:
\begin{enumerate}
\item[(i)] $\cW:=\{W_{\alpha}: \alpha\in \cA\}$ is a measurable partition into unstable curves;
\item [(ii)] There is a finite Borel measure $\nu$ supported on $\cW$ such that for any measurable set $B\subset M$,
$$
\nu(B)=\int_{\alpha\in\cA} \nu_{\alpha}(B\cap W_{\alpha})\, d\lambda(\alpha).
$$
\end{enumerate}
For simplicity, we denote such a family by $\cG=(\cW, \nu)$.

Given a standard family $\cG=(\cW, \nu)=\{(W_{\alpha}, \nu_\alpha),\ \alpha\in \cA, \ \lambda\}$ and $n\geq 1$,
the forward image family $T^n\cG=(\cW_n, \nu_n)=\{(V_\beta, \nu^n_\beta), \ \beta\in \cB, \ \lambda^n\}$
is such that each $V_\beta$ is a connected component of $T^nW_\alpha\backslash S_{-n}$
for some $\alpha\in \cA$,
associated with
$\nu^n_\beta(\cdot)=T^n_*\nu_\alpha(\cdot \ |\ V_\beta)$ (where $T^n_*\nu_{\alpha}$ is the pushforward of $\nu_{\alpha}$  to $T^n W_{\alpha}$ )
 and
$d\lambda^n(\beta)=\nu_\alpha(T^{-n}V_\beta) d\lambda(\alpha)$.
It is not hard to see that $T^n\cG$ is a standard family as well (cf. \cite{CM}).

Let $\fF$ be the collection of all standard families.
We introduce a characteristic function on $\fF$
to measure the average length of unstable curves in a standard family.
More precisely,
we define a function $\cZ: \fF\to [0, \infty]$ by
\beq\label{cZ}
\cZ(\cG)=\dfrac{1}{\nu(M)} \int_{\cA} |W_{\alpha}|^{-1} \, d\lambda(\alpha), \ \  \text{for any}\  \cG\in \fF.
\eeq
where $\bs_0$ is as in  assumption (\textbf{h2}).
Let $\fF_0$ denote the class of $\cG\in \fF$ such that $\cZ(\cG)<\infty$.

It turns out that the value of $\cZ(T^n\cG)$ decreases exponentially in $n$ until it becomes small enough.
This fundamental fact, called the growth lemma, was first proved by Chernov for dispersing billiards in \cite{C99},
and later proved in \cite{CM} under Assumptions (\textbf{h1})-(\textbf{h4}).

\begin{lemma}\label{lem: growth ch}
There exist constants $c>0$, $C_z>0$, and $\vartheta\in (0,1)$ such that for any $\cG\in \fF_0$,
$$
\cZ(T^n \cG)\leq c \vartheta^{n} \cZ(\cG)+C_z, \ \ \ \text{for any} \ n\ge 0.
$$
\end{lemma}

\bigskip

\section{Checking condition $\DD_q(u_n)$}

 The proof of Condition $\DD_q (u_n)$ uses decay of correlations and is fairly standard. The following analysis is taken from~\cite{FHN} but is included for completeness. The proof considers  stable manifolds rather than stable curves, as well as ideas from Young's Tower construction for billiard systems~\cite{Y98}. Recall that assumptions $(\textbf{h1})-$$(\textbf{h4})$ along with the Growth Lemma imply the existence of a Young Tower with exponential tails~\cite[Lemma 17]{CZ09}.  Throughout this section, to simplify notation, $C$ will denote an unspecified constant, whose value may vary from line to line.

We now briefly the structure of  a Young Tower with exponential return time  tails for  a map $T:M\to M$
of a  Riemannian   manifold  $M$ equipped with Lebesgue measure $m$.

 A Young Tower has  a base set $\Delta_0$ with a hyperbolic product structure as in Young~\cite{Y98}
 with  an $\mathscr{L}^1(m)$ return time function $R:\Delta_0\to\mathbb{N}$.
There is a countable partition $\Lambda_{0,i}$ of $\Delta_0$ so that
$R$ is constant on each partition element $\Lambda_{0,i}$. We denote $R|_{\Lambda_{0, i}}$ by $R_i$.
The  Young Tower is defined by
\[
\Delta = \bigcup_{i\in\N, 0\le l \le R_i-1}\{ (x, l) : x\in \Lambda_{0,i }\}
\]
equipped with a  tower map $F:\Delta \to \Delta $ given by
\[
F(x, l) = \begin{cases}(x, l+1) & \mbox{ if }  x\in \Lambda_{0,i}, l<R_i-1\\
(T^{R_i}x, 0) &\mbox{ if } x\in \Lambda_{0,i}, l = R_i-1\end{cases}.
\]
We will refer to $\Delta_0:= \cup_i(\Lambda_{0,i} ,0)$ as  the  base of the tower $\Delta$  and denote
 $\Lambda_i := \Lambda_{0,i}$. Similarly we call $\Delta_l=\{(x,l): l<R(x) \}$, the $l-$th level of the tower.
A key role is played by the return map $f=T^R:\Delta_0\to\Delta_0$ by $f(x)=T^{R(x)} (x)$, which is uniformly expanding  Gibbs-Markov.

We may form  a quotiented tower (see~\cite{Y98} for details) by introducing an equivalence relation for points on the same stable manifold.
This operation helps in our decay of correlations estimates, as it allows decay rates for the indicator function of complicated sets to be estimated in the $L^{\infty}$ norm.We now  list the features  of the Tower that we will use.

There exists an invariant measure $m_0$  for $f: \Delta_0\to \Delta_0$ which has
absolutely continuous conditional measures on local unstable manifolds in $\Delta_0$, with density bounded uniformly from above and below.

There exists an $F$-invariant measure $\nu$ on $\Delta$ which is given by
$\nu (B)=\frac{m_0(F^{-l} B)}{\int_{\Lambda_0} R\,dm_0 }$
 for a measurable $B\subset \Lambda_{l}$, and extended to the entire tower $\Delta$ in the obvious way.
There is a projection $\pi:\Delta \to M$ given by $\pi(x, l) = T^l(x)$ which semi-conjugates $F$ and $T$,
so that  $\pi\circ F = T\circ\pi$.
The invariant measure  $\mu$, which is an SRB measure for $T: M\to M$,  is given by $\mu = \pi_*\nu$.
Denote by $W^s(x)$ the local stable manifold through $x$ and by $W^u (x)$ the local unstable manifold.
Let $B(x,r)$  denote the ball of radius $r$ centered at the point $x$.

We lift a function $\phi: M \to \bR$ to $\Delta$ by defining  $\phi(x,l)=\phi(T^l x)$ (we keep the same symbol for $\phi$ on $\Delta$ and $\phi$ on $M$).

Under the assumption of exponential tails, that is if $m(R>n)=\mathcal{O}( \gamma_1^n)$ for some $0<\gamma_1<1$ then
Young shows~\cite{Y98} there exists $0<\Lambda_1<1$ such that for all Lipschitz $\phi$, $\psi$ we have
 \begin{equation}\label{Lip_Lip_decay}
 \left |\int \phi  \psi\circ T^n  d\mu -\int \phi d\mu \int \psi d\mu \right| \le C \Lambda_1^n \|\phi\|_{Lip} \|\psi \|_{Lip}
 \end{equation}
 for some constant $C$. Moreover, if  the lift of $\psi$ is constant on local stable leaves of the Young Tower,
 then
 \begin{equation}\label{Lip_infinity_decay}
 \left|\int \phi  \psi\circ T^n  \,d\mu -\int \phi \,d\mu \int \psi \,d\mu \right|
 \le C \Lambda_2^n \|\phi\|_{Lip} \|\psi \|_{\infty}.
 \end{equation}

\vspace{3mm}

 Let $D$ be a set whose boundary is piecewise smooth and finite length,  and define
\[
 H_{k}(D) = \left\{x \in D:T^{k}(W^s(x))\cap \partial D\not = \emptyset\right\}.
 \]

\begin{proposition}\label{prop:annulus1}
There exist constants $C>0$ and $0<\tau_1<1$ such that, for all $k$,
\begin{equation}\label{annulus}
 \mu(H_{ k}(D))\le  C \tau_1^{k}.
\end{equation}
\end{proposition}

\begin{proof}
 As a consequence of the uniform contraction of local stable manifolds, 
there exists $0<\tau_1<1$ and
$C_1>0$ such that $\dist\,(T^n(x), T^n(y))\le C_1\tau_1^n$ for all
 $y\in W^s(x).$ In particular, this implies that $|T^k(W^s(x))|\le C_1 \tau_1^k$. Therefore, for every  $x\in H_{k}(D)$,
 the stable manifold  $T^k(W^s(x))$ lies in an ``annulus'' of width $2 \tau_1^k$ around $\partial D$. By the invariance of $\mu$ the result follows.

\end{proof}

\begin{lemma}\label{lemma:dun-prelim}
Suppose
 $\Phi:M\to \bR$ is a Lipschitz map and $\Psi$ is the indicator function
 \[
\Psi:= \I_{\mathscr{W}_{0,\ell}\left(A^{(q)}_n\right)}.
 \]
Then for all $j\geq 0$, there exists $0<\Lambda_2<1$ such that
\begin{equation}
 \left|\int\Phi\,(\Psi\circ T^j)\, \text{d}\mu - \int\Phi\text{d}\mu\int\Psi\text{d}\mu\right|\le C\,\left(\|\Phi\|_\infty \,\,\tau_1^{\floor{j/2}}+\|\Phi\|_{\text{Lip}}\,\,\Lambda_2^{\floor{j/2}}\right).
\end{equation}
\end{lemma}

\begin{proof}

We choose for reference a point  $y \in \Delta_0$, where  $\Delta_0$ is the base of a Young Tower with a
  hyperbolic product structure as in Young~\cite{Y98}.  Let  the  local unstable manifold of $y$ be denoted $\tilde\gamma^u:=W^u(y)$. By the hyperbolic product structure of $\Delta$, each local stable manifold $W^s(x)$, $x\in \Delta_0$,  intersects $\tilde\gamma^u$ in a unique point $\hat{x}$.
  For every map $\Psi$ on $\Delta$ we define the function $\overline\Psi(x,r):=\Psi(\hat x,r)$.  The function $\overline \Psi_i$ is constant along stable manifolds. Young~\cite{Y98} has shown that  if $\overline\Psi\Delta \to \bR$ is constant on stable manifolds
 then there exists a constant $C$ and $0<\Lambda_2<1$ such that
 \begin{equation}\label{Lip_infinity_decay}
 \left|\int \phi \cdot \overline \psi\circ T^n  \,d\mu -\int \phi \,d\mu \int \psi \,d\mu \right|
 \le C \Lambda_2^n \|\psi\|_{Lip} \|\overline\psi \|_{\infty}.
 \end{equation}
 for any Lipschitz function $\phi: M\rightarrow  \bR$.


 Now we take  $ \Psi:= \I_{\mathscr{W}_{0,\ell}\left(A^{(q)}_n\right)}$ and $\Psi_i=\Psi\circ T^i$, with the  corresponding quotiented functions $\overline{\Psi}$ and
 $\overline{\Psi_i}$ which are constant on stable manifolds.

The set of points where $\overline \Psi_i\neq\Psi_i$ is, by definition, the set of $(x,r)$  for which there
exist  $x_1,x_2$ on the same local stable manifold as $T^r (x)$ such that
\[
x_1\in\mathscr{W}_{i,\ell}\left(A^{(q)}_n\right)
\]
 but
\[
x_2 \notin\mathscr{W}_{i,\ell}\left(A^{(q)}_n\right).
\]

This set is contained in $\cup_{k =i}^{i+\ell-1}H_k(A^{(q)}_n)$. If
we let $i\ge\floor{j/2}$ then
we  have
\[\mu(\set{\Psi_{i}\neq\overline\Psi_{i}}) \le
\sum_{k=\floor{j/2}}^{\infty} \mu(H_k(A^{(q)}_n))\le C \,\tau_1^{\floor{j/2}}.
\]
We also have
\[
 \abs{\int\Phi\,(\,\overline{\Psi}_{\floor{j/2}}\circ T^{j -\floor{j/2}})\,d\mu - \int\Phi d\mu\int \overline\Psi_{\floor{j/2}}d\mu}\le C\,\|\Phi\|_{\text{Lip}}\,\,\|\bar\Psi\|_{\infty}\,\,\Lambda_2^{\floor{j/2}}.
\]
for any Lipschitz function $\Phi$.
Using the identity
$$\int \phi\,(\psi\circ T)-\int \phi \int \psi = \int\phi \,(\psi\circ T-\bar{\psi}\circ T)+
\int \phi\,(\bar{\psi}\circ T) -\int \phi \int \bar{\psi} +\int \phi \int \bar{\psi} -\int \phi \int \psi$$
we obtain
\begin{align}
\Big|\int \Phi\,\left(\Psi_{\floor{j/2}}\circ T^{j -\floor{j/2}}\right)\,d\mu &- \int\Phi d\mu\int \Psi_{\floor{j/2}}\,d\mu\Big|\nonumber\\
%
&\le \abs{\int \Phi\, \left((\Psi_{\floor{j/2}} - \overline\Psi_{\floor{j/2}})\circ T^{j - \floor{j/2}}\right)\,d\mu}  +C \,\|\Phi\|_{\text{Lip}}\,\,\Lambda_2^{\floor{j/2}} \nonumber\\
&\quad + \abs{\int\Phi \,d\mu \int\left(\overline\Psi_{\floor{j/2}} - \Psi_{\floor{j/2}}\right)}
\,d\mu\nonumber\\
&\le C\left(2\,\|\Phi\|_\infty \,\,\mu\set{\overline\Psi_{\floor{j/2}}\neq\Psi_{\floor{j/2}}}+\|\Phi\|_{\text{Lip}}\,\,\Lambda_2^{\floor{j/2}}\right)\nonumber\\
&\le \cC\left(\|\Phi\|_\infty \,\, \tau_1^{\floor{j/2}}+\|\Phi\|_{\text{Lip}}\,\,\Lambda_2^{\floor{j/2}}\right).
\end{align}
We complete the proof by observing that $\int\Psi \,d\mu = \int\Psi_{\floor{j/2}} \,d\mu$, due to the $T$-invariance of $\mu$ and the fact that
 $\Psi_{\floor{j/2}}\circ T^{j -\floor{j/2}} = \Psi_{j} = \Psi\circ T^{j}.$
\end{proof}

To prove condition $\DD_q(u_n)$, we will approximate the characteristic function of the set $A^{(q)}_n$ by a suitable Lipschitz function. However, this Lipschitz function will decrease sharply to zero near the boundary of the set $A^{(q)}_n$. As the estimate in Lemma~\ref{lemma:dun-prelim} involves the Lipschitz norm, we need to bound its increase as we approach $\partial A^{(q)}_n$.

Let $A_n=A_n^{(q)}$ and $D_n:=\set{x\in A_n^{(q)}:\;\dist\left(x, \overline{A_n^c}\right)\geq n^{-2}},$
 where $\overline{A_n^c}$ denotes the closure of the complement of the set $A_n$.
Define $\Phi_n:\mathbb{X}\to\mathbb{R}$ as
\begin{equation}
\label{eq:Lip-approximation}
\Phi_n(x)=\begin{cases}
  0&\text{if $x\notin A_n$}\\
  \frac{\dist(x,A_n^c)}{\dist(x,A_n^c)+
  \dist(x,D_n)}&
  \text{if $x\in A_n\setminus D_n$}\\
  1& \text{if $x\in D_n$}
\end{cases}.
\end{equation}
Note that $\Phi_n$ is Lipschitz continuous with Lipschitz constant given by $n^2$. 
Furthermore. $\|\Phi_n-\I_{A_n}\|_{L^1(\mu)}\leq C/n^2$ for some constant $C$.

It follows that
\begin{align}
\Big|\int \I_{A_n^{(q)}}\,\left(\Psi_{\floor{j/2}}\circ T^{j -\floor{j/2}}\right) &\,d\mu - \mu(A_n^{(q)})\int \Psi d\mu\Big|\nonumber\\
&\le \abs{\int \left(\I_{A_n^{(q)}} - \Phi_n\right)\Psi_{\floor{j/2}}\,d\mu}+\cC \left(\|\Phi_n \|_\infty \,\,j^2\,\,\tau_1^{\floor{j/4}}+\|\Phi_n \|_{\text{Lip}}\,\,\Lambda_2^{\floor{j/2}}\right)\nonumber\\
&\quad+ \abs{\int \left(\I_{A_n^{(q)}} - \Phi_n \right)\,d\mu \int \Psi_{\floor{j/2}}\,d\mu},
\end{align}
and consequently
\[
 \abs{\mu\left(A_n^{(q)}\cap \mathscr W_{j,\ell}(A_n^{(q)})\right) - \mu(A_n^{(q)})\,\mu\left(\mathscr W_{0,\ell}(A_n^{(q)})\right)}\le \gamma(n,j)
\] where
\[
 \gamma(n,j) = C\,\left(n^{-2}+ n^{2} \,\theta_1^{\floor{j/2}}\right)
\] and
\[
\theta_1 = \max\,\set{\tau_1, \Lambda_2}.\]
Thus if, for instance, $j=t_n=(\ln n)^{1+\delta}$, then $n\gamma(n, t_n)\to 0$ as $n\to\infty.$  Note that we have considerable
freedom of choice of $j$ in order to ensure that the previous limit is zero; taking into account  our  future application, we chose $t_n=(\ln n)^{1+\delta}$.

\section{Checking condition $\DD_q'(u_n)$}

We now show that $\DD_q '(u_n)$ holds, which is in dynamical  applications  the more difficult condition to check.


We consider points which leave our neighborhood after the first iterate and return at the $j$th iterate. To simplify the notation and exposition we consider  $F=T^q$, so that $\zeta$ is a fixed point for $F$.
 Clearly for large $n$, $A_n^q=\{ X_0>u_n, X_0 \circ F <u_n\}$.

Note that for any $\delta >0$
$$\lim_{n\to\infty} n\sum_{j=(\ln n)^{1+\delta}}^{\lfloor n/k_n\rfloor}\mu ( A_n^q\cap T^{-j}A_n^q)=0.$$

The estimate above is immediate from Condition $\DD_q(u_n)$, taking $\ell=1$ in the equation
\[
 \abs{\mu\left(A_n^{(q)}\cap \mathscr W_{j,\ell}(A_n^{(q)})\right) - \mu(A_n^{(q)})\,\mu\left(\mathscr W_{0,\ell}(A_n^{(q)})\right)}\le \gamma(n,j)
\]
and noting $ \gamma(n,j) = \cC\,\left(n^{-2}+ n^{2} \,\theta_1^{\floor{j/2}}\right)$.

\bigskip

For $j\in [1,(\ln n)^{1+\delta})$, we will use the Growth lemma~\ref{lem: growth ch}, which states that  for any $\cG\in \fF_0$,
$$
\cZ(T^j \cG)\leq c \vartheta^{j} \cZ(\cG)+C_z, \ \ \ \text{for any} \ j\ge 0.
$$

For any standard family $\cG=(\cW, \nu)$,
any $x\in W\in \cW$ and $n\ge 0$,
we denote by $r_{\cG, n}(x)$ the distance from $T^n x$ to $\partial W_n$ measured along $W_n$,
where $W_n$ is the open connected component of $T^n W$ whose closure contains $T^n x$.
As a result of Lemma \ref{lem: growth ch},
we have the follow fact.

\begin{lemma}[cf. \cite{CM}]\label{lem: growth}
There exists $c_1>0$ such that for any standard family $\cG=(\cW, \nu)\in \fF_0$, any $\eps>0$ and $n\ge 0$,
we have
\beq
\nu(r_{\cG, n}(x)<\eps)\le (c_1 \cZ(\cG) + C_z) \eps.
\eeq
\end{lemma}

For the set $A_n^q$, we foliate it into unstable curves  that stretch completely from one side to the other. Let $\{W_{\beta}, \beta\in \cA, \lambda\}$ be the foliation, and $\lambda$ the factor measure defined on the index set $\cA$.  This enables us to define a standard family, denoted as $\cG_{n}=(\cW_{n},
\mu_{n})$, where $\mu_{n}:=\mu|_{A_n^q}/\mu(A_n^q)$. Clearly, $|W_{\beta}|\sim n^{-1/2}$ this is because for large  $n$, $A_n^q$ is a strip that has length $\sim n^{-1/2}$ and width $\sim n^{-1/2}$. Also, by our assumption, the periodic point is bounded away from the singularity set of the map,   so the density of $A_n^q$ is of order $1$. Thus we obtain for $j\leq \infty$,
\begin{align*}
\cZ(\cG_{n})&=\mu(A_n^q)^{-1}\int_{\beta\in\cA}|W_{\beta}|^{-1}\,\lambda(d\beta)\\
&\le C
\mu(A_n^q)^{-1}\cdot  n^{1/2}\cdot \mu(A_n^q) \le C n^{1/2}.\end{align*}

Using  Lemma \ref{lem: growth}, we define $\eps=n^{-1/2}$, thus we have that
\begin{align}\label{estMiMk1}
\mu(T^{-j}A_n^q\cap A_n^q)&\leq \mu(A_n^q) T^j_*\mu_{n}(r<\eps)=\mu(A_n^q)\mu_n(r_{\cG_n, j}(x)<\eps)\nonumber\\
&\leq \mu(A_n^q)(c_1 \cZ(T^j\cG_n) + C_z) \eps\nonumber\\
&\leq \mu(A_n^q)(c_1 c \vartheta^{j} \cZ(\cG) + C_z) \eps\nonumber\\
&=\mu(A_n^q)(c_1 c C \vartheta^{j} n^{1/2} + C_z) n^{-1/2}.
\end{align}
Hence
we have
\[
\lim_{n\to \infty} n\sum_{j=\ln\ln n}^{(\ln n)^{1+\delta}}\mu((A_n^q)\cap T^{-j}(A_n^q)) \le \lim_{n\to \infty} c_2\vartheta^{\ln\ln n}=0.
\]

\bigskip


 We consider points which leave our neighborhood after the first iterate and return at the $j$th iterate. To simplify the notation and exposition we consider  $F=T^q$, so that $\zeta$ is a fixed point for $F$.
 Clearly for large $n$, $A_n^q=\{ X_0>u_n, X_0 \circ F <u_n\}$.


\begin{figure}[h!]
	\begin{center}
	\begin{minipage}{.48\textwidth}
	\mbox{
		\hspace*{3.5em}{\begin{tikzpicture}
	\begin{pgfonlayer}{main}
	\draw[shift={(0,0)},fill=white] (-1,-1) rectangle (1,1);
	\node at (-0.50,0.70) {\small $\Omega(n)$};
	\end{pgfonlayer}
	\begin{pgfonlayer}{bkgr}
	\draw[shift={(0,0)},fill={rgb:red,51;green,102;blue,204},opacity=0.2] (-2,-1) rectangle (2,1);
	\draw[shift={(0,0)}] (-1,-1) rectangle (3,1);
	\end{pgfonlayer}
	\node (0,0) [below] {$\zeta$};
	\draw (0,0)--(1,0);
	\fill (0,0) circle [radius=1pt];
	\draw[<-] (1.5,-1)--(1.5,-1.5);
	\node at (1.5,-1.5) [below] {$A_1=\lbrace x\in \Omega(n), Tx\notin \Omega(n)\rbrace$};
	\node at (2.5,1) [above] {$TA_1$};
	\end{tikzpicture}}}
	\caption{$A_1$ and $TA_1$.}
	\end{minipage}
	\begin{minipage}{.48\textwidth}
		\mbox{
			\hspace*{1.5em}{\begin{tikzpicture}[scale=1.4]
				\draw rectangle (-2,2);
				\fill (-1,1) circle [radius=1pt];
				\node at (-.8,1) {$\zeta$};
				\draw (-1.75,1.5)--(-.25,1.5);
				\draw (-1.75,.55)--(-.25,.55);
				\draw (-2.25,1)--(.5,1);
				\node at (-2.75,.90) {$W^u(\zeta)$};
				\draw (-1,2.25) |- (-1,-.25);
				\node at (-1,2.25) {$W^s(\zeta)$};
				\node at (-.5,1.75) {$\Omega(n)$};
				\end{tikzpicture}}}
				\caption{$\Omega (n)$.}
\end{minipage}
\end{center}
\end{figure}

Recall that $DF_{\zeta}^u\sim\Lambda>1$ is the expansion in the unstable direction at the fixed point $\zeta$. If $W_{\beta}(x)$ is an unstable curve  for $x\in \Omega (n)$, the set $A_u (x):=A_n^q\cap W_{\beta} (x)$ has two connected components (right and left
subintervals) of $W_{\beta(x)}\cap \Omega (n)$ which are roughly a distance $\frac{1}{ \sqrt{n}}$ from $\zeta$.

For large $n$, on $\Omega (n)$ the map $F$ acts on $W^u (\zeta)\cap \Omega (n)$ as an expansion in the unstable direction for a  number of iterates. Furthermore if $m_u$ is Lebesgue measure on $ W^u (\zeta)$,
then $m_u ((A_u (x)\cap W^u (\zeta)) \cap F^j(A_u (x)\cap W^u (\zeta) ) )=0$ for all $0<j< \eta \ln n $ for some $\eta>0$. The same holds for
$W_{\beta(x)}\cap \Omega (n)$ if $n$ is sufficiently large. We will estimate an upper bound for $\eta$ for sufficiently large $n$ by solving $\Lambda^{k}/\sqrt{n}\sim n^{-\gamma}$ for some $\gamma\in(0,\frac{1}{2})$.  We see that $\eta \sim (1/2-\gamma)/(\ln \Lambda)$ will do. Given that we are using an adapted metric this implies that $\mu ((A_u (x)) \cap F^j (A_u (x)) )=0$ for all $0<j< \eta \ln n $.






Thus
\[
\lim_{n\to \infty} n\sum_{j=1}^{\eta\ln(n)}\mu((A_n^q)\cap T^{-j}(A_n^q)) =0.
\]

Thus condition $\DD_q'(u_n)$ follows.

\section{Applications  to Sinai dispersing billiards.}
 Our main applications are to Sinai dispersing billiards with finite and infinite horizon.
 We first consider Sinai dispersing billiards with finite horizon. A Poisson law was established for hitting time statistics in the papers~\cite{FHN,Pene_Saussol}
 for generic points (periodic points were explicitly not covered by these results).

 Let $\Gamma = \set{\Gamma_i, i = 1, \dots, k}$ be a family of pairwise disjoint, simply connected $C^3$ curves with strictly positive curvature
on the two-dimensional torus $\mathbb{T}^2$.  The billiard flow  $B_t$ is the dynamical system
generated  by the motion of  a point particle in $Q= \mathbb{T}^2/(\cup_{i=1}^k (\mbox{ interior } \Gamma_i))$ with constant unit velocity inside $Q$
and with elastic reflections at $\partial Q=\cup_{i=1}^k \Gamma_i$, where elastic means ``angle of incidence equals angle of reflection''.
If each $\Gamma_i$ is a circle then this system is called a periodic Lorentz gas, a classical model of statistical  physics. The billiard
flow is Hamiltonian and preserves  a probability measure (which is Liouville measure) $\tilde{\mu}$ given by
 $d\tilde{\mu}=C_{Q}\,dq\,dt$ where $C_{Q}$ is a normalizing constant and $q\in Q$, $t\in \bR$ are Euclidean coordinates.

In this paper we consider  the billiard map  $T: \partial Q \to \partial Q$ rather than the flow.  Let $r$ be a one-dimensional coordinatization of
$\Gamma$ corresponding to length and let $n(r)$ be the outward normal to $\Gamma$ at the point $r$. For each
$r\in \Gamma$ we consider the tangent
space at $r$ consisting of unit vectors $v$ such that $(n(r),v)\ge 0$. We identify each such unit vector $v$ with an
angle $\vartheta \in [-\pi/2, \pi/2]$. The boundary $M$ is then parametrized by $M:=\partial Q=\Gamma\times [-\pi/2, \pi/2]$  so that $M$ consists of  the
points $(r,\vartheta)$. $T:M\to M$ is the Poincar\'e map that gives the position and angle $T(r,\vartheta)=(r_1,\vartheta_1)$  after a point $(r,\vartheta)$
flows under $B_t$ and collides again with $M$, according to  the rule  angle of incidence equals angle of reflection. Thus if
$ (r,\vartheta)$ is the time of flight before collision $T(r,\vartheta)=B_{h (r,\vartheta)} (r,\vartheta)$.
The billiard map preserves a probability measure $d\mu=c_r \cos \vartheta\, dr d\vartheta$ equivalent to the
$2$-dimensional Lebesgue measure $dm=dr\,d\vartheta$ with
density $\rho (x)=c_r\cos\vartheta$ where $x=(r,\vartheta)$ and $c_{r}$ is a normalizing constant to ensure that the invariant measure is a probability measure i.e. has total mass one.

 The time of flight, $h: \partial Q\to \bR$ by  $h(x,r)= \min \{ t>0: B_t (x,r) \in \partial Q \}$ is  the flow time it takes for a point on the boundary of $Q$ to return to the
 boundary.  If the time of flight $h(r,\theta)$  is uniformly bounded above then we say the billiard system has finite horizon.  Young~\cite{Y98} proved that finite horizon
 Sinai dispersing   billiard maps have exponential decay of correlations for H\"{o}lder observations.
Chernov~\cite{C99} extended this result to planar dispersing billiards with piecewise $C^3$ smooth
boundaries and where the flight time $h(x,r)$ can become singular along  a countable numbers of smooth curves.

The assumptions $\textbf{(h1)}-\textbf{(h4)}$ are satisfied by Sinai dispersing billiards with either finite or infinite horizon.

We let $\zeta$ be a periodic point of period $q$  whose trajectory does not hit a singularity curve.

We define a dynamically adapted metric $d$ on $M$ as in (\ref{defn:d}), such that $B_r(\zeta):=d^{-1}(\zeta, r)$ is a  hyperbolic d-ball with radius $r$, centered at $\zeta$.


We let $\phi:M\to \mathbb{R}\cup\{+\infty\}$ be given by $$\phi(z)=-\ln (d(z,\zeta))$$

Let $X_n:=\phi\circ T^n$. Since $\mu$ is invariant, so the process $\{X_n, n\geq 0\}$ is stationary.  Note that $\{X_n>u\}=T^{-n}(B_{e^{-u}}(\xi))$.  In order to obtain a nondegenerate limit for the distribution of $M_n=\max\{X_0,\cdots, X_n\}$, we choose a normalizing sequence $\{u_n\}$  as  such that  $$u_n=u_n(\tau):=\inf \{x\in M\,:\, \mu(X_0\leq x)\geq 1-\frac{\tau}{n}\}$$
for any $\tau>0$.
Then we can check that (\ref{chooseun}) holds:
\beq\lim_{n\to\infty} n \mu(X_0>u_n)=\tau
\eeq

We now define $\theta=1-\frac{1}{|DT^q_u (\zeta)I} \in (0,1)$, where $DT^q_u (\zeta)$ is the derivative if $T^q$ in the unstable direction at $\zeta$. Then one can check that
$$\theta=\lim_{n\to\infty} \mu(T^{-q} B_{e^{-u_n}}(\xi)|B_{e^{-u_n}}(\xi))$$
Then
\beq\label{theta1}
\lim_{n\to\infty} \mu(X_q> u_n| X_0>u_n)=\lim_{n\to\infty} \mu(T^{-q} B_{e^{-u_n}}(\zeta)|B_{e^{-u_n}}(\zeta))=\theta\eeq

Now we can apply our main theorem 2 to get the compound Poisson Law with parameter $\theta$ for hitting
time statistics while Theorem 1 yields a Gumbel law with extremal index $\theta$ for the maximal process
$M_n=\max \{\phi,\phi\circ T,\ldots, \phi\circ T^{n-1} \}$.

\subsection{A calculation of extremal indices in the infinite horizon case.}

Suppose $(T,X,\mu)$ is a Sinai dispersing billiard with infinite horizon. By considering the associated billiard flow we see that there are infinitely
many periodic orbits $\zeta_n$ all of  period $2$ as indicated in Figure 3. These periodic orbits limit on the grazing point $p$ in Figure 3. In this section we calculate the extremal index of $\theta_n$ of $\zeta_n$ and show $\lim_{n\to \infty} \theta_n=0$.

\begin{figure}[h!]
	\centering
	
		\includegraphics[width = \textwidth]{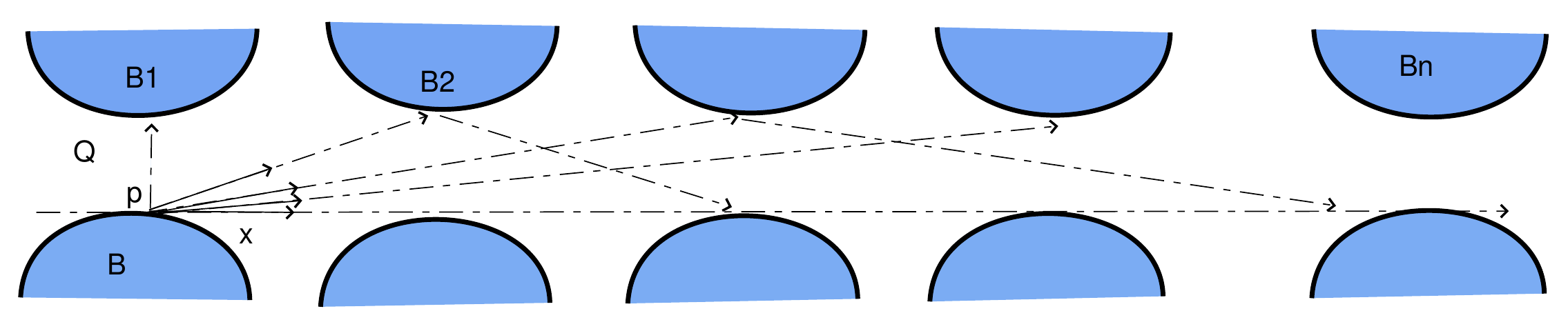}
		\caption{Periodic orbits in infinite horizon case.}
	
\end{figure}

Let $B^u(\zeta_n)$ be the curvature of the unstable wave front, see \cite[Chapter 3]{CM} for a precise definition. By time reversibility, we know that $B^u(\zeta_n)=B^u(T\zeta_n)$. By \cite{CM}~(3.34), writing $\zeta_n=(0,\vartheta_n)$,
$$B^u(\zeta_n)=\frac{1}{\tau(\zeta_n)+\frac{1}{R(\zeta_n)+\frac{1}{B^u(\zeta_n)}}}$$
where $\tau(\zeta_n)$ is the free path of the periodic point $\zeta_n$, and $R(\zeta_n)=\frac{2K(\zeta_n)}{\cos\vartheta_n}$ is the collision factor, $K(\zeta_n)$ is the curvature of the base point of $\zeta_n$. Solving the above equation, we get
$$B^u(\zeta_n)=R(\zeta_n)(\sqrt{1+\frac{4}{R(\zeta_n)\tau(\zeta_n)}}-1)$$

For large $n$, $R(\zeta_n)\sim n$ and $\tau(\zeta_n)\sim n$, which implies that
$$B^u(\zeta_n)=\frac{2}{\tau(\zeta_n)}+ O(n^{-2})$$

 By \cite[Theorem 3.38]{CM},
$$\log |DT_u (\zeta_n)=\frac{1}{\tau(\zeta_n)}\int_0^{\tau(\zeta_n) } B^u(\Phi^t \zeta_n)\,dt=\frac{1}{\tau(\zeta_n)}\log(1+B^u(\zeta_n\tau(\zeta_n)),$$
where \cite{CM}~(3.32) is used in  the  estimation of the last step.
By the  time reversibility property of billiards, $|DT_u^2(\zeta_n)|=|DT_u(\zeta_n)|^2$, where $DT_u (\zeta_n)$ is the derivative of $T$ in the unstable direction at $\zeta_n$.

Let $\theta_n$ be the extremal index of the periodic point $\zeta_n$ of period $2$.
Since $1-\theta_n=\frac{1}{|DT^2_u (\zeta_n)|}$ we may calculate
\begin{align*}
1-\theta_n &=(1+\tau(\zeta_n)B^u(\zeta_n))^{-\frac{2}{\tau(\zeta_n)}}\\
&=\left(1+\tau(\zeta_n)R(\zeta_n)(\sqrt{1+\frac{4}{R(\zeta_n)\tau(\zeta_n)}}-1)\right)^{-\frac{2}{\tau(\zeta_n)}}\\
&=\left(1+\frac{4}{\sqrt{1+\frac{4}{R(\zeta_n)\tau(\zeta_n)}}+1}\right)^{-\frac{2}{\tau(\zeta_n)}}\\
\end{align*}
For $n$ large enough, we have
$$\theta_n=1-3^{-\frac{2}{\tau(\zeta_n)}}+ O(n^{-\frac{1}{n}})$$
Thus $$\lim_{n\to\infty}\theta_n=0$$.
This implies that as $n$ gets larger, the extreme index of periodic point $\xi_n$ is decreasing.

\end{document}